\documentclass[12pt, a4paper]{article}
\usepackage{latexsym, amsmath, amsfonts, amsthm, amssymb}
\usepackage{times}
\usepackage{a4wide}
\usepackage{graphicx, tocloft}
\usepackage{mathrsfs}
\usepackage{url, hyperref}

\def \tr {\text{\rm Tr}}

\def \Vol {{ \rm Vol }}

\def \RR {\mathbb R}

\def \EE {\mathbb E}

\def \eps {\varepsilon}
\def \vphi {\varphi}

\newtheorem{theorem}{Theorem}[section]

\newtheorem{lemma}[theorem]{Lemma}
\newtheorem{example}[theorem]{Example}

\newtheorem{corollary}[theorem]{Corollary}

\theoremstyle{definition}
\newtheorem{remark}[theorem]{Remark}

\def\myffrac#1#2 in #3{\raise 2.6pt\hbox{$#3 #1$}\mkern-1.5mu\raise 0.8pt\hbox{$
		#3/$}\mkern-1.1mu\lower 1.5pt\hbox{$#3 #2$}}

\def\qed{\hfill $\vcenter{\hrule height .3mm
		\hbox {\vrule width .3mm height 2.1mm \kern 2mm \vrule width .3mm
			height 2.1mm} \hrule height .3mm}$ \bigskip}

\def \id {{\rm Id}}

\def \cov {{ \rm Cov}\,}

\def \Var {{ \rm Var}\,}

\begin{document}

\title{Digesting the proof of the sharp thin-shell inequality}
\author{Yuansi Chen and Boaz Klartag}
\date{}
\maketitle

\begin{abstract} We present a proof that determines  the optimal value of the universal constant in the thin-shell theorem for log-concave distributions 
in high dimensions. We prove that for any log-concave random vector $X = (X_1,\ldots,X_n)$ in $\RR^n$ with mean zero and identity covariance,
$$ \Var( |X|^2 ) \leq 8 n. $$
The constant $8$ is optimal: equality is attained when $X_1,\ldots,X_n$ are independent, identically distributed, standard, centered exponential random variables.
Moreover, among isotropic random vectors distributed uniformly on convex bodies in $\RR^n$, the quantity $\Var(|X|^2)$ is maximized by the uniform distribution on a regular simplex.
We also provide a corresponding sharp bound on the Hilbert-Schmidt norm of the tensor of $3^{rd}$-moments of isotropic, log-concave distributions. 

\medskip The argument  relies on the analysis of a weighted Riemannian manifold associated with log-concave moment measures and the Monge-Amp\`ere equation.
This manifold was studied in this context in \cite{lc_moment}. The main improvement over \cite{lc_moment} comes from a concise 
yet effective analysis of the $3^{rd}$-derivatives tensor of the potential. 

\medskip The proof was found by GPT-5.6 Pro in response to prompts supplied by the first-named author, following general discussions between the two authors concerning log-concave moment measures. The prompts referred 
to the paper ``Logarithmically-concave moment measures I'' and suggested bootstrapping a bound on the second trace moment.
\end{abstract}

\section{Introduction}
\label{sec1}

A random vector $X$ in $\RR^n$ is log-concave if it is supported 
in a convex set $K \subseteq \RR^n$ with density $\rho: K \rightarrow (0,\infty)$ 
such that $\log \rho$ is concave. For instance, the uniform probability measure on a bounded convex set is log-concave, as is 
any Gaussian measure. The random vector $X$ is isotropic if $\EE X = 0$ 
and $\cov(X) = \id$, where
$$ \cov(X) = \left( \EE [X_i X_j] - \EE [X_i] \EE [X_j] \right)_{i,j=1,\ldots,n} \in \RR^{n \times n} $$
is the covariance matrix. A log-concave random vector has moments of all orders, and hence its expectation and covariance are well-defined.
It was conjectured by Anttila, Ball and Perissinaki \cite{ABP} that most of the mass of an isotropic, log-concave 
random vector $X$ in $\RR^n$ is concentrated in a thin spherical shell, whose width is much smaller than its radius. Bobkov and Koldobsky \cite{BK} formulated 
the variance conjecture\footnote{The original formulations in \cite{ABP, BK} were focused on uniform distributions on convex bodies rather than on the slightly more 
general case of log-concave distributions.}
\begin{equation}
\Var(|X|) \leq \frac{1}{n} \EE \left( |X|^2 - n \right)^2 \leq C, \label{eq_1818} 
\end{equation}
where $C > 0$ is a universal constant, and where the first inequality follows from the fact that $(r - \sqrt{n})^2 \leq (r^2 -n)^2 / n$ for all $r \geq 0$.
 The motivation for these  conjectures
stems from the fact that thin-shell bounds lie at the heart of the proof of the central limit theorem for convex bodies. 
See e.g. \cite{K_euro} for a survey. An additional motivation comes from the connection to Bourgain's slicing 
problem from Eldan and Klartag \cite{EK}. 

\medskip The variance conjecture (\ref{eq_1818}) was recently proven by Klartag and Lehec \cite{KL_thin},
following a long chain of developments: the middle expression in (\ref{eq_1818}) 
was shown to be at most $C n^{2/5 + o(1)}$ in \cite{K_poly}, at most $C n^{3/8}$ in Fleury \cite{fleury},
at most  $C n^{1/3}$ in Gu\'edon and Milman \cite{guedon_milman}, at most $C n^{1/4}$ in Lee and Vempala \cite{LV_focs},
at most 
 $ C \exp \left( ( \log n)^{1/2 + o(1)} \right) = n^{o(1)} $
 in \cite{chen}, at most $C \log^4 n$ in \cite{KL}, at most $C \log^{2.23} n$ in Jambulapati, Lee and Vempala \cite{JLV},
 at most $C \sqrt{\log n}$ in \cite{K_root} and at most $C \log \log n$ in Guan \cite{guan}.
 All of these arguments rely either on concentration of measure on the
high-dimensional sphere or on ideas related to Eldan's stochastic
localization \cite{eldan1}. In this paper we use a different method, 
 involving Monge-Amp\`ere equations and log-concave moment measures, and prove the following:
 
 \begin{theorem} For any isotropic, log-concave random vector $X = (X_1,\ldots,X_n)$ in $\RR^n$,
$$ \EE \left( |X|^2 - n \right)^2   \leq 8 n. $$
Equality holds when $X_1,\ldots,X_n$ are independent, identically distributed
random variables with density
$e^{-x-1}\mathbf{1}_{\{x\geq-1\}}$, i.e., standard, centered, exponential
random variables.
\label{thm1}
 \end{theorem}

The argument also gives a sharp estimate for the $3^{rd}$-moment
tensor. For a random vector $X$ in $\RR^n$ with finite third moments, write
$$ T = T(X) = \left(\EE[X_iX_jX_k]\right)_{i,j,k=1}^n \in \RR^{n \times n \times n}. $$
The Hilbert-Schmidt norm of $T$ is given by 
$$  \|T \|_{HS}^2
= \sum_{i,j,k=1}^n
\left(\EE[X_iX_jX_k]\right)^2.
$$

\begin{theorem}
\label{thm2}
Let $X$ be an isotropic log-concave random vector in $\RR^n$. Then, for $T = T(X)$,
$$ 
\| T \|_{HS}^2\leq 4n.
$$
Equality is attained when the coordinates
of $X$ are independent standard, centered, exponential random variables.
\end{theorem}

\medskip In the case of uniform distributions on convex bodies,
Theorem \ref{thm1} and Theorem \ref{thm2} yield the following corollary, proved
by passing to a cone in one dimension higher.

\begin{corollary}
\label{cor1}
Let $X$ be an isotropic random vector distributed uniformly on a convex
body in $\RR^n$. Then
$$
\Var(|X|^2)
\leq
\frac{4n(n+1)^2}{(n+3)(n+4)},
$$ 
and with $T = T(X)$,
$$
\| T \|_{HS}^2
\leq
\frac{4n(n-1)(n+2)}{(n+3)^2}.
$$
Both estimates are sharp: equality is attained when $X$ is distributed
uniformly on a regular simplex in isotropic position.
\end{corollary}

\medskip Corollary \ref{cor1} gives some positive evidence for  the non-symmetric Mahler conjecture, whose $2$-dimensional case 
was already proven  in \cite{mahler}, and whose $3$-dimensional case was recently established 
by Chen, Li, Xi and Xu \cite{CLLX}. 
Recall that if $X$ is distributed uniformly on a convex body $K\subseteq\RR^n$, then the isotropic constant of $K$ is
an affine invariant of $K$ defined via
$$ L_K =
\frac{\det(\cov(X))^{1/(2n)}}{\Vol_n(K)^{1/n}}, $$
where $\Vol_n$ stands for $n$-dimensional volume.
The strong slicing conjecture asserts that $L_K\leq L_{\Delta^n}$,
where $\Delta^n$ is any simplex in $\RR^n$. The variance conjecture has
been considered stronger than the slicing problem \cite{EK}, and its
sharp resolution in Theorem \ref{thm1} suggests that perhaps the strong
slicing conjecture is within reach. It was shown in \cite{K_mahler}
that the strong slicing conjecture implies the non-symmetric Mahler
conjecture. The {\it functional version} of the strong slicing problem suggests that for any isotropic, log-concave random vector $X$ in $\RR^n$,
$$
h(X) \geq n,
$$
where $h(X) = -\int_{\RR^n} \rho \log \rho$ is the differential entropy of $X$, and $\rho$ is the density of $X$. 
Equality is attained when the coordinates of $X$ are independent standard, centered, exponential random variables.
The strong slicing conjecture follows from its functional version by the same cone construction as in the
deduction of Corollary \ref{cor1} that is described in Lemma \ref{lem_1433} below. Thus, a possible route towards the non-symmetric Mahler conjecture is to extend the moment-measure argument of the present paper from the second and third moments to entropy. 

\medskip 
Let us pass from random vectors to probability measures. Let $\mu$ be the law of the isotropic log-concave random vector $X$
in $\RR^n$.
For a function $f\in L^2(\mu)$ with $\int f d\mu = 0$, we denote
$$
\|f\|_{H^{-1}(\mu)}
=
\sup\left\{
\int_{\RR^n}fg\,d\mu
\,;\,
g\in\mathcal{C}_c^\infty(\RR^n),\
\int_{\RR^n}|\nabla g|^2\,d\mu\leq1
\right\},
$$
where $\mathcal{C}_c^\infty(\RR^n)$ is the space of smooth, compactly supported functions in $\RR^n$.
The $H^{-1}$-norm is related to the infinitesimal quadratic  cost associated with transporting the perturbation $(1+\eps f)\mu$
to the original measure $\mu$. 
As explained e.g. in \cite{KL_thin},  the $H^{-1}$-inequality from \cite{ptrf} and Barthe and Klartag \cite{Barthe_klartag} implies that
\begin{equation}
\Var (|X|^2)
\leq
4\sum_{i=1}^n\|x_i\|_{H^{-1}(\mu)}^2.
\label{eq_1212}
\end{equation}
Thus, Theorem \ref{thm1} follows from the following theorem, which was established in \cite{KL_thin} with a non-optimal universal constant. 

\begin{theorem} Let $\mu$ be the law of the isotropic log-concave random vector $X$
in $\RR^n$. Then,
$$ \sum_{i=1}^n\|x_i\|_{H^{-1}(\mu)}^2 \leq 2 n. $$
Equality is attained when the coordinates
of $X$ are independent standard, centered, exponential random variables.
\label{thm3}
\end{theorem}

 The proofs of the above results  rely on log-concave moment measures, 
which we now briefly describe. Let $\mu$ be a centered probability measure whose support spans $\RR^n$. 
The moment measure theorem was proven by Cordero-Erausquin and Klartag \cite{CK},
building on work of Berman and Berndtsson \cite{BB} and Wang and Zhu \cite{WZ}.
It states that there exists an essentially-continuous convex function $\psi:\RR^n\rightarrow\RR\cup\{+\infty\}$, uniquely determined up to translations, such that $\mu$ is the {\it moment measure} of $\psi$. That is, the measure $\nu$ on $\RR^n$ defined via
\begin{equation}
d\nu(x)=e^{-\psi(x)}\,dx
\label{eq_1703} \end{equation}
is a probability measure, and the gradient map $\nabla\psi$ pushes forward $\nu$  to $\mu$.
Our main object of study is the matrix field 
\begin{equation}
 H(x) = \nabla^2 \psi(x). \label{eq_1704} \end{equation}
Under mild regularity assumptions, the symmetric matrix $H(x)$
is well-defined, positive-definite and depends smoothly on $x \in \RR^n$. 
 The main technical contribution of this paper is the following
estimate:

\begin{theorem} Let $X$ be an isotropic, log-concave random vector in $\RR^n$, and
assume that its law satisfies the regularity assumptions from \cite{lc_moment}. Then,
$$ \int_{\RR^n} | H(x) |^2 d \nu(x) \leq 2n, $$
where $|A|^2 = \| A \|_{HS}^2$ is the square of the Hilbert-Schmidt norm of the matrix $A \in \RR^{n \times n}$. 
\label{thm4}
\end{theorem}

In Section \ref{sec4} below, we explain that Theorem \ref{thm3} follows from Theorem \ref{thm4} 
by a simple integration by parts, which is essentially equivalent to the fact that $H \circ (\nabla \psi)^{-1}$ 
is a Stein kernel. 
Section \ref{sec2} contains background on log-concave moment measures. 
Section \ref{sec3} contains the $3^{rd}$-moment computations by GPT-5.6 Pro, 
as well as the proofs of Theorem \ref{thm2} and Theorem \ref{thm4}.
Below, we write $\nabla^2 f(x) \in \RR^{n \times n}$ for the Hessian matrix 
of the function $f$ at the point $x \in \RR^n$. We write $x \cdot y = \langle x,y \rangle$ for the standard scalar product of  $x,y \in \RR^n$,
and $|x| = \sqrt{x \cdot x}$.

\medskip {\it Acknowledgements}. The first-named author thanks Ronen Eldan for providing access to GPT Pro
and for encouraging its use. He also thanks Alessio Figalli for discussions
of the Monge--Amp\`ere equation and for providing references on the subject,
and Tristan Matsulevits for discussions. This project began during an extended visit by the second-named author to the ETH Institute for Theoretical Studies, and he is grateful to the Institute for its hospitality. The second-named author is supported by a grant from the Israel Science Foundation (ISF).

\medskip {\it AI use statement}. GPT-5.6 Pro produced the initial proof in response to prompts from the first-named author. Both authors verified and revised the argument, and rewrote it to make it more accessible. GPT-5.6 was then used to polish the writing.

\section{Background on log-concave moment measures}
\label{sec2}

We begin the proof with a description of log-concave moment measures and their basic properties.
This remarkable construction associates with any centered, full-dimensional log-concave probability measure a weighted Riemannian manifold possessing a form of {\it uniform convexity}. More precisely, its Bakry-\'Emery Ricci tensor is bounded from below by one half of the metric tensor, while its optimal Poincar\'e constant equals one; see \cite{lc_moment}. Moreover, the linear functions form an $n$-dimensional space of eigenfunctions of $-L$ corresponding to the eigenvalue one. These properties make the construction particularly suitable for the analysis of high-dimensional volume distribution in convex sets, since many of the relevant questions become easier in the presence of uniform convexity.

\medskip Let $X$ be an isotropic, log-concave random vector in $\RR^n$ with law $\mu$. As in \cite{lc_moment}, let us make the regularity assumption that
$\mu$ is supported in a bounded, open, convex set $K\subseteq\RR^n$ and has density $\rho:K\rightarrow(0,\infty)$, where $V = -\log \rho$ is a smooth convex function such that $V$ and all its 
partial derivatives of all orders are bounded. Under this assumption, the convex function $\psi:\RR^n\rightarrow\RR$ whose moment measure is $\mu$ is smooth and strictly convex. Moreover, the map $\nabla\psi:\RR^n\rightarrow K$ is a diffeomorphism. 
Since $\nabla \psi$ pushes forward $\nu$  to $\mu$, the change-of-variables formula yields the Monge--Amp\`ere equation
\begin{equation}
\rho(\nabla\psi(x))\det\nabla^2\psi(x)
=
e^{-\psi(x)}
\qquad (x\in\RR^n).
\label{eq_1142}
\end{equation}
In the particular case where $\mu$ is uniformly distributed in a convex body $K$, the function $\rho$ is constant in $K$, and \eqref{eq_1142} is the toric K\"ahler--Einstein equation, up to an additive normalization of $\psi$. In the general case, the {\it weighted Riemannian manifold} associated with the log-concave measure $\mu$
is
\begin{equation} M_\mu^*
=
\left(\RR^n,\nabla^2\psi, \nu \right).
\label{eq_1733} \end{equation}
That is, we equip $\RR^n$ with the Riemannian metric tensor
$$ 
g=\sum_{i,j=1}^n\psi_{ij}\,dx^i dx^j
$$
and with the probability measure $\nu$ satisfying $d\nu(x)=e^{-\psi(x)}\,dx$, as in \eqref{eq_1703}. Here, we write
$$
\psi_i=\partial_i\psi,
\qquad
\psi_{ij}=\partial_i\partial_j\psi,
\qquad
\psi_{ijk}=\partial_i\partial_j\partial_k\psi.
$$
In fact, (\ref{eq_1733}) is the ``complex coordinates'' description of this weighted Riemannian manifold.
The term ``complex coordinates'' comes from toric K\"ahler geometry. There is an equivalent description in ``action coordinates'', which we now describe. Let
Let $\varphi=\psi^*$ be the Legendre transform of $\psi$, so that
$$
\varphi(y) =\psi^*(y) = \sup_{x \in \RR^n} \left[ x \cdot y - \psi(x) \right].
$$
Then $\vphi: K \rightarrow \RR$ is a smooth function. In fact, the two maps
$$
\nabla\psi:\RR^n\longrightarrow K
\qquad \textrm{and} \qquad 
\nabla\varphi:K\longrightarrow\RR^n
$$
are inverse to one another, and with $H = \nabla^2 \psi$ as in (\ref{eq_1704}) we have
$$
\nabla^2\varphi(y)
=
H(\nabla\varphi(y))^{-1}.
$$
The corresponding weighted Riemannian manifold in action coordinates is
$$
M_\mu
=
\left(K,\nabla^2\varphi,\mu\right),
$$
and the map $\nabla\psi$ is an isomorphism between $M_\mu^*$ and $M_\mu$. The manifold $M_\mu$ has several useful properties that were proven in \cite{lc_moment}. It satisfies the sharp Poincar\'e inequality
\begin{equation}
\Var_\mu(f)
\leq
\int_K
\sum_{i,j=1}^n \varphi^{ij}f_if_j\,d\mu,
\label{eq_1743}
\end{equation}
for any smooth function $f \in L^2(\mu)$, where $(\varphi^{ij})=(\nabla^2\varphi)^{-1}$
and $\Var_{\mu}(f) = \int f^2 d \mu - (\int f d \mu)^2$. Equality holds in (\ref{eq_1743}) when $f$ is an affine function. Equivalently, in complex coordinates,
\begin{equation}
\Var_\nu(u)
\leq
\int_{\RR^n}\sum_{i,j=1}^n\psi^{ij}u_i u_j\,d\nu.
\label{eq_1744}
\end{equation}

\medskip Most of our calculations will be carried out in complex coordinates. The weighted Laplacian is
$$
Lu
=
\sum_{i,j=1}^n\psi^{ij}u_{ij}
-
\sum_{i=1}^n(V_i\circ\nabla\psi)u_i.
$$
Its defining property is that, for any smooth functions $u,v: \RR^n \rightarrow \RR$ with at least one of them compactly supported, 
\begin{equation}
\int_{\RR^n}(Lu)v\,d\nu
=
-\int_{\RR^n}\sum_{i,j=1}^n\psi^{ij}u_i v_j\,d\nu
=
-\int_{\RR^n} \Gamma(u, v)  \,d\nu,
\label{eq_integration_by_parts_L}
\end{equation}
where for smooth functions $u,v: \RR^n \rightarrow \RR$ we set
$$
\Gamma(u,v)
=
\sum_{i,j=1}^n\psi^{ij}u_i v_j = \langle \nabla_g u, \nabla_g v \rangle_g,
\qquad
\Gamma(u)=\Gamma(u,u).
$$
Here,  $\langle \nabla_g u, \nabla_g v \rangle_g$ is the Riemannian scalar product between the Riemannian gradients of $u$ and $v$ with respect to the metric
 $g$. The notation $\Gamma(u,v)$ is standard in $\Gamma$-calculus; see e.g. Bakry, Gentil and Ledoux \cite{BGL}.
Consider the symmetric bilinear form
$$
\mathcal{E}_0(u,v)
=
\int_{\RR^n} \Gamma(u,v)\,d\nu  = 
\int_{\RR^n}\sum_{i,j=1}^n\psi^{ij}u_i v_j\,d\nu,
\qquad \qquad
u,v\in\mathcal{C}_c^\infty(\RR^n).
$$
It follows from (\ref{eq_integration_by_parts_L}) that the Dirichlet form $\mathcal{E}_0$ is closable 
in $L^2(\nu)$. We denote its closure by $(\mathcal{E},D(\mathcal{E}))$.
The weighted Riemannian manifold $M_\mu^*$ has yet another useful property: it is
stochastically complete, see \cite{lc_moment}. We will not use this fact below. Instead, all
integrations by parts on $M_\mu^*$ are justified by using the cutoff functions
constructed in Appendix \ref{sec_cutoffs}. Theorem 1.1 in \cite{lc_moment} yields the pointwise bound
\begin{equation}
\tr H(x)\leq 2R(K)^2,
\qquad
R(K)=\sup_{y\in K}|y|.
\label{eq_trace_bound}
\end{equation}
Below, we only use the weaker conclusion of (\ref{eq_trace_bound}): that the matrix field $H$ is bounded.

\medskip The isotropic normalization of $\mu$ has a particularly simple interpretation in terms of this matrix field:
for  $i,j=1,\ldots,n$,
\begin{equation}
\int_{\RR^n}H_{ij}\,d\nu
=
\int_{\RR^n}\psi_i\psi_j\,d\nu
=
\int_{\RR^n}y_i y_j\,d\mu(y)
=
\delta_{ij}.
\label{eq_mean_H}
\end{equation}
The first equality is a standard integration by parts, which for completeness is justified in Lemma
\ref{lem_euclidean_cutoff_identities} below. The other equalities follow from
the push-forward relation $(\nabla\psi)_\#\nu=\mu$ and the isotropicity of
$\mu$. Consequently,
$$
\int_{\RR^n}H\,d\nu=\id.
$$

\section{The new $3^{rd}$-order tensor calculation}
\label{sec3}

In this section we prove  
Theorem \ref{thm2} and Theorem 
\ref{thm4}.
Lemma 5.2 in \cite{lc_moment} states, in the notation of the present paper, that
\begin{equation}
LH+H\geq 0
\label{eq_1323} \end{equation}
in the sense of symmetric matrices (when $L$ is applied to a matrix field as in (\ref{eq_1323}), it acts entrywise).
Let us begin with an exact identity refining inequality (\ref{eq_1323}). We use the Einstein summation convention in the calculations below. Define the symmetric matrix fields
$$
A
=
H\left(\nabla^2V\circ\nabla\psi\right)H,
\qquad
Q_{ij}
=
\tr\left(H^{-1} (\partial_i H) H^{-1} (\partial_j H) \right).
$$

\begin{lemma}
The symmetric matrices $A$ and $Q$ are positive semi-definite, and
\begin{equation}
LH+H=A+Q.
\label{eq_LH_exact}
\end{equation}
\label{lem_LH_exact}
\end{lemma}

\begin{proof}
Taking the logarithm in \eqref{eq_1142}, we obtain
$$
\psi-V(\nabla\psi)+\log\det H=0.
$$
Differentiating twice and using
$$
\partial_i\partial_j\log\det H
=
\psi^{ab}\partial_i\partial_jH_{ab}
-
\psi^{ac}\psi^{bd}(\partial_i H)_{cd}(\partial_j H)_{ab},
$$
we obtain
\begin{align*}
\psi^{ab}(H_{ij})_{ab}
-
(V_k\circ\nabla\psi) \partial_k H_{ij}
+
H_{ij}
=
H_{ia}(V_{ab}\circ\nabla\psi)H_{bj} +
\psi^{ac}\psi^{bd} (\partial_i H)_{cd} (\partial_j H)_{ab}.
\end{align*}
This is precisely \eqref{eq_LH_exact}. The matrix $A$ is positive semi-definite by the convexity of $V$. If $v\in\RR^n$ and $H_v=\sum_i v_i \partial_i H$, then
$$
\sum_{i,j=1}^n v_iQ_{ij}v_j
=
\tr\left[\left(H^{-1/2}H_vH^{-1/2}\right)^2\right]
\geq0.
$$
Thus $Q$ is also positive semi-definite.
\end{proof}

Denote
$$
N_2
=
\int_{\RR^n}\tr(H^2)\,d\nu.
$$
At a fixed point $x\in\RR^n$, choose Euclidean coordinates in which
the symmetric, positive-definite matrix $H(x)$ is a diagonal matrix. We may thus write
$$
H(x)=\operatorname{diag}(\lambda_1,\ldots,\lambda_n),
$$
with $\lambda_i > 0$ for all $i$. Define
\begin{align}
d_2
&=
\psi^{ab}\tr(\partial_a H \partial_b H)
=
\sum_{a,i,j=1}^n\frac{\psi_{aij}^2}{\lambda_a},
\label{eq_d2}\\
a_2
&=
\tr(HA),
\\
q_2
&=
\tr(HQ)
=
\sum_{a,i,j=1}^n
\frac{\lambda_a}{\lambda_i\lambda_j}\psi_{aij}^2.
\label{eq_q2}
\end{align}
These definitions are coordinate invariant, and all three quantities are non-negative. We also write
$$
D_2=\int_{\RR^n}d_2\,d\nu,
\qquad
A_2=\int_{\RR^n}a_2\,d\nu,
\qquad
Q_2=\int_{\RR^n}q_2\,d\nu.
$$

\begin{lemma}
Pointwise in $\RR^n$,
\begin{equation}
L\tr(H^2)
=
2\left[-\tr(H^2)+d_2+a_2+q_2\right].
\label{eq_pointwise_trace}
\end{equation}
\label{lem_pointwise_trace}
\end{lemma}

\begin{proof}
The Leibniz rule for $L$ and \eqref{eq_LH_exact} give
\begin{align*}
L(H_{ij}H_{ij})
&=
2H_{ij}LH_{ij}
+
2\psi^{ab}(\partial_aH_{ij})(\partial_bH_{ij})\\
&=
2\left[-\tr(H^2)+\tr(HA)+\tr(HQ)\right]
+
2\psi^{ab}\tr(\partial_a H \partial_b H).
\end{align*}
This is \eqref{eq_pointwise_trace}.
\end{proof}

\begin{lemma}
The functions $d_2,a_2,q_2$ are integrable and
\begin{equation}
N_2=D_2+A_2+Q_2.
\label{eq_integrated_trace}
\end{equation}
\label{lem_integrated_trace}
\end{lemma}

\begin{proof} By \eqref{eq_trace_bound}, the function 
$$ G = \tr(H^2) $$
is bounded, since $0\leq G\leq(\tr H)^2$.
Since $d_2, a_2, q_2 \geq 0$, the desired conclusion (\ref{eq_integrated_trace}) 
follows from Lemma \ref{lem_pointwise_trace} once we prove that $d_2+a_2+q_2$ is $\nu$-integrable and
\begin{equation} 
\int_{\RR^n}LG \,d\nu  = 
\int_{\RR^n}L\tr(H^2)\,d\nu=0.
\label{eq_1430}
\end{equation}
It remains to prove \eqref{eq_1430}.
To this end, we need to verify the conditions of Lemma \ref{lem_integrating_L_cutoff}.
Let us first show that
\begin{equation}
\Gamma(G)
\leq 4Gd_2. \label{eq_1433} 
\end{equation}
Indeed, $\partial_a G = \partial_a \tr(H^2) = 2 \tr(H \partial_a H)$.
At the fixed point $x$, using the coordinates where $H(x) = \operatorname{diag}(\lambda_1,\ldots,\lambda_n)$, we have $\partial_a H_{ii} = \psi_{aii}$ and
$$
\partial_a G = 2 \sum_{i=1}^n \lambda_i \psi_{aii}.
$$
The Cauchy-Schwarz inequality yields
$$
(\partial_a G)^2
=
4 \left( \sum_{i=1}^n \lambda_i \psi_{aii} \right)^2
\leq
4 \left( \sum_{i=1}^n \lambda_i^2 \right) \left( \sum_{i=1}^n \psi_{aii}^2 \right)
=
4 G \sum_{i=1}^n \psi_{aii}^2.
$$
Hence, 
$$
\Gamma(G)
=
\sum_{a=1}^n \frac{(\partial_a G)^2}{\lambda_a}
\leq
4 G \sum_{a,i=1}^n \frac{\psi_{aii}^2}{\lambda_a},
$$
proving (\ref{eq_1433}). Setting 
$$ B=\|G\|_\infty, \qquad S=d_2+a_2+q_2
$$
we deduce from (\ref{eq_1433}) 
that 
\begin{equation}
\Gamma(G) \leq 4BS.
\label{eq_gradient_G}
\end{equation}
In view of \eqref{eq_pointwise_trace} and \eqref{eq_gradient_G}, the functions
$G$ and $S$ satisfy the assumptions of Lemma
\ref{lem_integrating_L_cutoff} in Appendix \ref{sec_cutoffs}. It follows that
$S$ is $\nu$-integrable and (\ref{eq_1430}) holds true. 
\end{proof}

\begin{lemma}
We have
$$
N_2-n\leq D_2.
$$
\label{lem_poincare_H}
\end{lemma}

\begin{proof}
By Lemma \ref{lem_integrated_trace}, the  Dirichlet energy of the entries of $H$ is finite:
$$
\sum_{i,j=1}^n \int_{\RR^n} \Gamma(H_{ij})\, d \nu = 
\sum_{i,j=1}^n
\int_{\RR^n}
\psi^{ab}(\partial_aH_{ij})(\partial_bH_{ij})\,d\nu
=
D_2.
$$
The entries of $H$ are bounded by \eqref{eq_trace_bound}. Lemma
\ref{lem_energy_cutoffs} in Appendix \ref{sec_cutoffs} therefore shows that
$H_{ij}\in D(\mathcal{E})$ for all $i,j$, and that we may apply \eqref{eq_1744} to
each entry of $H$. Thus, for all $i,j$,
$$ \Var_\nu(H_{ij}) \leq \int_{\RR^n}
\psi^{ab}(\partial_aH_{ij})(\partial_bH_{ij})\,d\nu. $$
Consequently, using \eqref{eq_mean_H} we obtain
\begin{align}
N_2-n
=
\int_{\RR^n}\|H-\id\|_{HS}^2\,d\nu
=
\sum_{i,j=1}^n\Var_\nu(H_{ij})
\leq
D_2. \tag*{\qedhere}
\end{align}
\end{proof}

The main new observation enabling the bootstrap is contained in the following elementary identity.

\begin{lemma}
Pointwise in $\RR^n$,
$$
q_2-d_2
=
\frac{1}{6}
\sum_{a,i,j=1}^n
\frac{\psi_{aij}^2}{\lambda_a\lambda_i\lambda_j}
\left[
(\lambda_a-\lambda_i)^2
+
(\lambda_i-\lambda_j)^2
+
(\lambda_j-\lambda_a)^2
\right]
\geq0.
$$
Consequently, $Q_2\geq D_2$.
\label{lem_cyclic_identity}
\end{lemma}

\begin{proof}
It follows from \eqref{eq_d2} and \eqref{eq_q2} that
$$
q_2-d_2
=
\sum_{a,i,j=1}^n
\frac{\lambda_a^2-\lambda_i\lambda_j}
{\lambda_a\lambda_i\lambda_j}\,
\psi_{aij}^2.
$$
The denominator is invariant under cyclic permutations of $(a,i,j)$. 
Since the tensor $(\psi_{aij})_{a,i,j=1,\ldots,n}$ is symmetric and the sum is over all ordered triples, we may replace the numerator by its cyclic average
$$
\frac{1}{3}
\left(
\lambda_a^2+\lambda_i^2+\lambda_j^2
-
\lambda_a\lambda_i-\lambda_i\lambda_j-\lambda_j\lambda_a
\right).
$$
This is one sixth of
$$
(\lambda_a-\lambda_i)^2
+
(\lambda_i-\lambda_j)^2
+
(\lambda_j-\lambda_a)^2,
$$
and the lemma follows.
\end{proof}

\begin{proof}[Proof of Theorem \ref{thm4}]
From Lemmas \ref{lem_integrated_trace} and \ref{lem_cyclic_identity}, and since $A_2\geq0$, we obtain
$$
N_2
=
D_2+A_2+Q_2
\geq
D_2+Q_2
\geq
2D_2.
$$
On the other hand, Lemma \ref{lem_poincare_H} yields $D_2\geq N_2-n$. Consequently,
$$
N_2\geq2(N_2-n),
$$
and hence
\begin{equation}
\int_{\RR^n}\tr(H^2)\,d\nu
=
N_2
\leq
2n.
\label{eq_H_second_trace}
\end{equation}
\end{proof}

\begin{example} {\rm Let us consider 
the example of the 
 standard, centered, exponential distribution
 in dimension one.  In this case,
$$
\psi(t)=e^t-t,
\qquad
d\nu(t)=e^{-e^t+t}\,dt,
\qquad
H(t)=e^t.
$$
The change of variables $s=e^t$ gives
$$
\int_\RR H(t)^2\,d\nu(t)
=
\int_0^\infty s^2e^{-s}\,ds
=
2.
$$
Similarly, by considering $n$ independent 
copies of the standard, centered, exponential distribution, we obtain an isotropic, log-concave measure $\mu$ in $\RR^n$ so that with the above notation,
$$ \int_{\RR^n} \| H(x) \|_{HS}^2 d \nu(x) = 2n. $$
}
\end{example}

\bigskip
We proceed with analysis of the third-moment tensor. 
For $i,j,k=1,\ldots,n$ set
$$
T_{ijk}=\int_{\RR^n}y_i y_j y_k\,d\mu(y).
$$

\begin{lemma}
For any $i,j,k=1,\ldots,n$,
$$
T_{ijk}
=
2\int_{\RR^n}\psi_{ijk}\,d\nu
=
2\int_{\RR^n}(H_{ij}-\delta_{ij})\psi_k\,d\nu.
$$
\label{lem_third_moment_cubic}
\end{lemma}

\begin{proof}
Observe first that \eqref{eq_trace_bound} and $D_2<\infty$ imply that every entry of the tensor $\nabla^3\psi$ belongs to $L^2(\nu)$. Indeed, if $M=2R(K)^2$, then $H\leq M\id$ and hence $H^{-1}\geq M^{-1}\id$. Therefore,
$$
d_2
=
\sum_{i,j=1}^n
\left\langle H^{-1}\nabla H_{ij},\nabla H_{ij}\right\rangle
\geq
\frac{1}{M}
\sum_{a,i,j=1}^n\psi_{aij}^2.
$$
Thus the additional integrability assumption in Lemma \ref{lem_euclidean_cutoff_identities} is satisfied; we know that $\psi_{ijk}\in L^2(\nu)\subseteq L^1(\nu)$, while $\nabla\psi$ and $H$ are bounded. Hence Lemma \ref{lem_euclidean_cutoff_identities} in Appendix \ref{sec_cutoffs} gives
$$
\int_{\RR^n}\psi_i\psi_j\psi_k\,d\nu
=
2\int_{\RR^n}\psi_{ijk}\,d\nu.
$$
Since $(\nabla\psi)_\#\nu=\mu$, the left-hand side equals $T_{ijk} = \int_{\RR^n}y_i y_j y_k\,d\mu(y)$, which proves the first equality. 
The second equality follows from \eqref{eq_third_mean_identity} and the fact that $\int_{\RR^n}\psi_k\,d\nu=0$.
\end{proof}

\begin{proof}[Proof of Theorem \ref{thm2}]
Assume first that the regularity 
assumptions from \cite{lc_moment}, 
which are also described  in Section \ref{sec2}, hold true. 
By Lemma \ref{lem_third_moment_cubic},
$$
T_{ijk}
=
2\int_{\RR^n}(H_{ij}-\delta_{ij})\psi_k\,d\nu.
$$
The functions $\psi_1,\ldots,\psi_n$ are orthonormal in $L^2(\nu)$. Consequently, for any fixed $i,j$,
$$
\sum_{k=1}^n
\left(
\int_{\RR^n}(H_{ij}-\delta_{ij})\psi_k\,d\nu
\right)^2
\leq
\int_{\RR^n}(H_{ij}-\delta_{ij})^2\,d\nu.
$$
Summing over $i,j$ and using \eqref{eq_H_second_trace}, we obtain
\begin{align*}
\|T\|_{HS}^2
\leq
4\int_{\RR^n}\|H-\id\|_{HS}^2\,d\nu
=
4(N_2-n)
\leq
4n.
\end{align*}
This proves Theorem \ref{thm2} under the regularity assumptions.

\medskip Let us now explain how the regularity assumptions may be removed. Let $\mu$ be an arbitrary isotropic log-concave probability measure. For $\delta,\varepsilon,R>0$, let $\widetilde{\mu}_{\delta,\varepsilon,R}$ be the probability measure whose density is proportional to
$$
\mathbf{1}_{B(0,R)}(x)
\exp\left(-\frac{\varepsilon|x|^2}{2}\right)
\frac{d(\mu*\gamma_\delta)}{dx}(x),
$$
where $\gamma_\delta$ is the centered Gaussian measure with covariance
$\delta\id$. In the interior of $B(0,R)$, the negative logarithm of this
density is smooth and uniformly convex, and all of its derivatives are
bounded. After centering and applying an invertible linear map, the resulting
measure is isotropic and satisfies all of the regularity assumptions used
above. We may choose $\delta\rightarrow0$, $\varepsilon\rightarrow0$ and
$R\rightarrow\infty$ along a diagonal sequence so that
$\widetilde{\mu}_{\delta,\varepsilon,R}$ converges to $\mu$ in all moments of
order at most four. Indeed, convergence of moments under Gaussian convolution
follows by writing it as the law of $X+\sqrt{\delta}Z$, while the Gaussian
factor and the truncation may then be removed by dominated convergence. The
means and covariance matrices of these approximating measures converge to zero
and $\id$, respectively. Consequently, their centered, isotropically normalized
images also converge to $\mu$ in all moments of order at most four. Once proved
under the regularity assumptions, the estimate in
Theorem \ref{thm2} therefore passes to the limit.

\medskip 
For the equality case, let $X$ have independent standard, centered,
exponential coordinates. All third moments involving at least two distinct coordinates vanish, while
$\EE X_i^3=2$ for every $i$. Hence
$$
\|T(X)\|_{HS}^2
=
\sum_{i=1}^n\left(\EE X_i^3\right)^2
=
4n.
$$
Thus equality is attained, and the proof of Theorem \ref{thm2} is complete.
\end{proof}

\section{Proofs of the main results} 
\label{sec4}

We first transfer the estimate of Theorem \ref{thm4} from the moment
coordinates (or ``complex coordinates'') to the original log-concave measure (or ``action coordinates''). 
This allows us to connect the Monge--Amp\`ere calculation with the
thin-shell problem.

\begin{proof}[Proof of Theorem \ref{thm3}]
Assume first that $\mu$ satisfies the regularity assumptions from Section~\ref{sec2}. Let
$$
S=(\nabla\psi)^{-1}:K\longrightarrow\RR^n
$$
and define the matrix field
$$
\tau(y)=H(S(y))
\qquad (y\in K).
$$
For $g\in\mathcal{C}_c^\infty(\RR^n)$ and $i=1,\ldots,n$, integration by
parts and the push-forward relation $(\nabla\psi)_\#\nu=\mu$ give
\begin{align}
\int_K y_i g(y)\,d\mu(y)
&=
\int_{\RR^n}\psi_i(x)g(\nabla\psi(x))\,d\nu(x)\nonumber\\
&=
\int_{\RR^n}\sum_{j=1}^n
H_{ij}(x)(\partial_jg)(\nabla\psi(x))\,d\nu(x)\nonumber\\
&=
\int_K\sum_{j=1}^n\tau_{ij}(y)\partial_jg(y)\,d\mu(y).
\label{eq_moment_stein}
\end{align}
To justify the second equality, one may insert the Euclidean cutoff
$\zeta_R$ used in the proof of Lemma
\ref{lem_euclidean_cutoff_identities} and integrate the derivative of
$\zeta_R g(\nabla\psi)e^{-\psi}$. The cutoff term tends to zero because
$\nabla\psi(\RR^n)=K$ is bounded and
$\|\nabla\zeta_R\|_\infty=O(R^{-1})$. The remaining terms are dominated
since $H$ is bounded.

\medskip Changing variables once more and using Theorem \ref{thm4}, we obtain
$$
\int_K\|\tau(y)\|_{HS}^2\,d\mu(y)
=
\int_{\RR^n}\tr(H(x)^2)\,d\nu(x)
\leq 2n.
$$
It follows from \eqref{eq_moment_stein} and the Cauchy--Schwarz inequality that
$$
\left|\int_Ky_i g(y)\,d\mu(y)\right|^2
\leq
\left(\int_K\sum_{j=1}^n\tau_{ij}^2\,d\mu\right)
\left(\int_K|\nabla g|^2\,d\mu\right).
$$
Taking the supremum over $g$ and then summing over $i$ yields
$$
\sum_{i=1}^n\|x_i\|_{H^{-1}(\mu)}^2
\leq
\int_K\|\tau\|_{HS}^2\,d\mu
\leq2n.
$$
Let us now remove the regularity assumptions. Let $(\mu_m)$ be the isotropic
regular approximations constructed at the end of Section \ref{sec3}. Thus,
$\mu_m$ converges weakly to $\mu$. As explained in the proof of Theorem 1.4
in \cite{KL_thin}, the functional
$$
\mu\longmapsto
\sum_{i=1}^n\|x_i\|_{H^{-1}(\mu)}^2
$$
is lower semi-continuous under weak convergence within the class of
isotropic, log-concave probability measures. Consequently,
\begin{equation}
\sum_{i=1}^n\|x_i\|_{H^{-1}(\mu)}^2
\leq
\liminf_{m\rightarrow\infty}
\sum_{i=1}^n\|x_i\|_{H^{-1}(\mu_m)}^2
\leq 2n.
\label{eq_123} \end{equation}
For the equality case, suppose that the coordinates of $X$ are independent
standard, centered, exponential random variables. In this case,
$$
\Var(|X|^2)=8n.
$$
It follows from \eqref{eq_1212} and \eqref{eq_123} that
$$
8n
\leq
4\sum_{i=1}^n\|x_i\|_{H^{-1}(\mu)}^2
\leq
8n.
$$
Thus equality holds in the conclusion of the present theorem.
\end{proof}

\begin{remark}
Under the regularity assumptions, the matrix field $\tau$ in
\eqref{eq_moment_stein} is the moment-map Stein kernel of $\mu$.  Thus
Theorem \ref{thm4} gives a Stein kernel whose squared Hilbert--Schmidt norm
has integral at most $2n$. Such a connection between moment maps and Stein kernels was previously
observed by Fathi \cite{Fathi}. To make the present argument as simple as possible, only the
integration by parts \eqref{eq_moment_stein} and Cauchy--Schwarz inequalities
are used.
\end{remark}

\begin{proof}[Proof of Theorem \ref{thm1}]
Combining \eqref{eq_1212} with Theorem \ref{thm3}, we obtain
$$
\Var(|X|^2)
\leq
4\sum_{i=1}^n\|x_i\|_{H^{-1}(\mu)}^2
\leq8n.
$$
To see that the constant is optimal, let $E$ be a standard exponential
random variable and set $Z=E-1$. Since $\EE E^r=r!$ for
$r=1,2,3,4$,
$$
\EE Z=0,
\qquad
\EE Z^2=1,
\qquad
\EE Z^4
=24-4\cdot6+6\cdot2-4+1
=9.
$$
If $Z_1,\ldots,Z_n$ are independent copies of $Z$, then
$X=(Z_1,\ldots,Z_n)$ is isotropic and log-concave, and
$$
\Var(|X|^2)
=
\sum_{i=1}^n\Var(Z_i^2)
=8n.
$$
\end{proof}

We next record a standard cone construction, in a form adapted
to the second and third moments. The underlying conical integration formula
appears, for example, in \cite[Lemma 2.1]{K_mahler}.

\begin{lemma}
\label{lem_1433}
Let $X$ be an isotropic random vector that is distributed uniformly in a convex body
$K\subseteq\RR^n$. Put $k=n+1$, let $G$ be a random variable independent of $X$ with density
$$
\frac{s^{k-1}e^{-s}}{(k-1)!}\mathbf{1}_{\{s>0\}},
$$
and define the random vector $Y\in\RR^k$ by
\begin{equation}
Y
=
\left(
\frac{GX}{\sqrt{k(k+1)}},
\frac{G-k}{\sqrt{k}}
\right).
\label{eq_cone_vector}
\end{equation}
Then $Y$ is isotropic and log-concave. Moreover,
\begin{equation}
\Var(|Y|^2)
=
\frac{k+3}{k(k+1)}
\left[
(k+2)\Var(|X|^2)+4k^2
\right],
\label{eq_cone_variance}
\end{equation}
and
\begin{equation}
\|T(Y)\|_{HS}^2
=
\frac{(k+2)^2}{k(k+1)}\|T(X)\|_{HS}^2
+12-\frac{8}{k}.
\label{eq_cone_third}
\end{equation}
\end{lemma}

\begin{proof}
Consider the convex cone
$$
\mathcal{C}
=
\left\{(sx,s)\in\RR^{n+1}\,;\ x\in K,\ s>0\right\}.
$$
The change of variables $(x,s)\mapsto(sx,s)$ has Jacobian $s^n$. Hence
$(GX,G)$ has density
$$
s^{-n} \cdot \frac{e^{-s} s^n}{n!\Vol_n(K)}\mathbf{1}_{\mathcal{C}}(z,s) = 
\frac{e^{-s}}{n!\Vol_n(K)}\mathbf{1}_{\mathcal{C}}(z,s)
$$
with respect to Lebesgue measure on $\RR^{n+1}$. This density is
log-concave, and therefore so is its affine image $Y$.
The moments of $G$ are
\begin{align*}
\EE G&=k,
&
\EE G^2&=k(k+1),\\
\EE G^3&=k(k+1)(k+2),
&
\EE G^4&=k(k+1)(k+2)(k+3).
\end{align*}
Since $\EE X=0$ and $\EE[X_iX_j]=\delta_{ij}$, these identities show
directly that $\EE Y=0$ and $\cov(Y)=\id$. Thus $Y$ is isotropic.
Write $U=|X|^2$ and $v=\Var(U)$. Since $\EE U=k-1$,
$$
|Y|^2
=
\frac{G^2U}{k(k+1)}
+
\frac{(G-k)^2}{k}.
$$
We need to compute $\EE |Y|^4$. Since
$$
\EE U^2=v+(k-1)^2, 
\qquad
\EE|Y|^2=k,
$$
and 
$$
\EE\!\left[G^2(G-k)^2\right]=k(k+1)(k+6),
\qquad
\EE(G-k)^4=3k(k+2),
$$
and since $U$ is independent of $G$, we have 
\begin{align*}
\EE|Y|^4
&=
\frac{(k+2)(k+3)}{k(k+1)}
\left[v+(k-1)^2\right]
+
\frac{2(k-1)(k+6)}{k}
+
\frac{3(k+2)}{k}\\
&=
k^2+
\frac{k+3}{k(k+1)}
\left[(k+2)v+4k^2\right].
\end{align*}
Subtracting $(\EE|Y|^2)^2=k^2$ proves
\eqref{eq_cone_variance}.
For the third moments, let $1\leq i,j,\ell\leq k-1$. Independence and the
above formulae for the moments of $G$ yield
\begin{align*}
T(Y)_{ij\ell}
&=
\frac{k+2}{\sqrt{k(k+1)}}T(X)_{ij\ell},\\
T(Y)_{ij\,k}
&=
\frac{2}{\sqrt{k}}\delta_{ij},\\
T(Y)_{i\,k\,k}
&=0,
\qquad
T(Y)_{k\,k\,k}
=
\frac{2}{\sqrt{k}}.
\end{align*}
The tensor is symmetric, so the second line occurs in three positions.
Consequently,
$$
\|T(Y)\|_{HS}^2
=
\frac{(k+2)^2}{k(k+1)}\|T(X)\|_{HS}^2
+
\frac{12(k-1)}{k}
+
\frac{4}{k},
$$
which is \eqref{eq_cone_third}.
\end{proof}

\begin{proof}[Proof of Corollary \ref{cor1}]
Apply Theorem \ref{thm1} in dimension $k=n+1$ to the random vector from
Lemma \ref{lem_1433}. Equation \eqref{eq_cone_variance} gives
$$
\frac{k+3}{k(k+1)}
\left[
(k+2)\Var(|X|^2)+4k^2
\right]
\leq8k.
$$
It follows that
$$
\Var(|X|^2)
\leq
\frac{4k^2(k-1)}{(k+2)(k+3)}
=
\frac{4n(n+1)^2}{(n+3)(n+4)}.
$$
Similarly, Theorem \ref{thm2} and \eqref{eq_cone_third} imply
$$
\frac{(k+2)^2}{k(k+1)}\|T(X)\|_{HS}^2
+12-\frac{8}{k}
\leq4k,
$$
and hence
$$
\|T(X)\|_{HS}^2
\leq
\frac{4(k-1)(k-2)(k+1)}{(k+2)^2}
=
\frac{4n(n-1)(n+2)}{(n+3)^2}.
$$
It remains to verify sharpness. Suppose that $K$ is a simplex with vertices
$v_1,\ldots,v_k$, centered and in isotropic position. Let
$E_1,\ldots,E_k$ be independent standard exponential random variables, write
$E=(E_1,\ldots,E_k)$ and put $G=\sum_{r=1}^kE_r$. The vector
$(E_1/G,\ldots,E_k/G)$ is independent of $G$ and is distributed uniformly
on the standard simplex
$$
\left\{x=(x_1,\ldots,x_k)\in\RR^k\,;\,
\sum_{i=1}^k x_i=1,\  x_i\geq0\ \textrm{for all }i
\right\}.
$$
Hence the random vector $\sum_{r=1}^k(E_r/G)v_r$ has the same law as $X$.
Thus,
$$
(GX,G)
\ \stackrel{d}{=}\
\sum_{r=1}^k E_r(v_r,1),
$$
where $\stackrel{d}{=}$ denotes equality in distribution.
Let $L:\RR^k\rightarrow\RR^k$ be the linear map whose $r^{th}$ column is
$(v_r,1)$, and let
$$
D
=
\begin{pmatrix}
[k(k+1)]^{-1/2}\id_{k-1}&0\\
0&k^{-1/2}
\end{pmatrix}.
$$
Since $\sum_rv_r=0$, the vector in \eqref{eq_cone_vector} has the same law
as $DL(E-\mathbf{1})$, where $\mathbf{1} = (1,\ldots,1)$. Both $Y$ and $E-\mathbf{1}$ are isotropic, and hence
$DL$ is orthogonal. Thus $Y$ is an orthogonal image of a product of
standard centered exponential variables. Equality holds in Theorems
\ref{thm1} and \ref{thm2}, and therefore in both estimates above. Finally,
a simplex in isotropic position is a regular simplex up to an orthogonal
transformation. This completes the proof.
\end{proof}

\appendix
\section{Justifying integrations by parts using cutoff functions}
\label{sec_cutoffs}

In this appendix we justify the integrations by parts used in Sections
\ref{sec2}, \ref{sec3} and \ref{sec4}. 
Throughout the appendix, $\mu,\nu,\psi,V,H,L$ and $\Gamma$ are as in Section
\ref{sec2}. In particular,
$$
d\nu=e^{-\psi}\,dx,
\qquad
H=\nabla^2\psi,
\qquad
\nabla\psi(\RR^n)=K,
$$
where $K$ is bounded.

\begin{lemma}
There exist functions $\chi_R\in\mathcal{C}_c^\infty(\RR^n)$, for all
sufficiently large $R$, such that
\begin{equation}
\begin{split}
&0\leq\chi_R\leq1,
\qquad
\chi_R\longrightarrow1
\quad\textrm{in $L^2(\nu)$ and $\nu$-almost everywhere},\\
&\mathcal{E}(\chi_R,\chi_R)\leq\frac{C A}{R},
\end{split}
\label{eq_energy_cutoffs}
\end{equation}
where $A<\infty$ depends only on $V$ and $K$. Consequently, if $F$ is a
bounded, smooth function satisfying
\begin{equation}
\int_{\RR^n}\Gamma(F)\,d\nu<\infty,
\label{eq_finite_energy}
\end{equation}
then $F\in D(\mathcal{E})$, and its Dirichlet energy equals the integral in
\eqref{eq_finite_energy}, and $F$ satisfies the Poincar\'e inequality \eqref{eq_1744}.
\label{lem_energy_cutoffs}
\end{lemma}

\begin{proof}
Since $\psi$ is convex and $\int_{\RR^n}e^{-\psi}=1$, the function $\psi$ tends to infinity at
infinity (e.g., \cite[Lemma 2.2.1]{BGVV}). 
 Thus all sublevel
sets of $\psi$ are compact. Moreover,
$$
L\psi
=
n-\left\langle\nabla V(\nabla\psi),\nabla\psi\right\rangle
\leq A,
$$
where
$$
A
=
\max\left\{
0,\,
n-\inf_{y\in K}\langle\nabla V(y),y\rangle
\right\}
<\infty.
$$
Let $\eta:\RR\rightarrow[0,1]$ be a smooth non-increasing function such that
$$
\eta(t)=1\quad\textrm{for }t\leq1,
\qquad
\eta(t)=0\quad\textrm{for }t\geq2,
$$
and set
$$
\chi_R(x)=\eta\left(\frac{\psi(x)}{R}\right).
$$
For large $R$, the function $\chi_R$ has compact support, since it vanishes
outside the compact set $\{\psi\leq2R\}$. Moreover, $\chi_R\rightarrow1$
pointwise and in $L^2(\nu)$.
For $t\in\RR$, put
$$
q_R(t)
=
\frac{1}{R}
\int_{t/R}^\infty\eta'(s)^2\,ds.
$$
Then $q_R$ is smooth and non-negative,  and
$$
-q_R'(t)
=
\frac{1}{R^2}\eta'(t/R)^2.
$$
The function $q_R(\psi)$ is compactly supported. Hence the 
integration by parts  \eqref{eq_integration_by_parts_L} gives
\begin{align*}
\mathcal{E}(\chi_R,\chi_R)
&=
\int_{\RR^n}
\frac{1}{R^2}\eta'(\psi/R)^2\Gamma(\psi)\,d\nu\\
&=
-\int_{\RR^n}q_R'(\psi)\Gamma(\psi)\,d\nu\\
&=
\int_{\RR^n}q_R(\psi)L\psi\,d\nu\\
&\leq
A\|q_R\|_\infty
\leq
\frac{CA}{R},
\end{align*}
where $C=\int_\RR\eta'(s)^2\,ds$. This proves
\eqref{eq_energy_cutoffs}.
Suppose now that $F$ is a bounded, smooth function satisfying \eqref{eq_finite_energy}. Then
$\chi_RF\in\mathcal{C}_c^\infty(\RR^n)$, and
\begin{align}
\int_{\RR^n}\Gamma\!\left((1-\chi_R)F\right)d\nu
&\leq
2\int_{\RR^n}(1-\chi_R)^2\Gamma(F)\,d\nu
+
2\|F\|_\infty^2\mathcal{E}(\chi_R,\chi_R)
\longrightarrow0.
\label{eq_form_cutoff_approximation}
\end{align}
Also, $(1-\chi_R)F\rightarrow0$ in $L^2(\nu)$. Moreover, 
\begin{align*}
\mathcal{E}_0((\chi_R-\chi_S)F,(\chi_R-\chi_S)F)
&\leq
2\int_{\RR^n}(\chi_R-\chi_S)^2\Gamma(F)\,d\nu\\
&\quad+
4\|F\|_\infty^2
\left[
\mathcal{E}(\chi_R,\chi_R)+\mathcal{E}(\chi_S,\chi_S)
\right],
\end{align*}
and the right-hand side tends to zero as $R,S\rightarrow\infty$. Hence
$(\chi_RF)$ is Cauchy in the Dirichlet-form norm. 
The limit of $\chi_RF$ in $L^2(\nu)$ is $F$, and the closedness of the
Dirichlet form therefore implies that $F\in D(\mathcal{E})$.
Equation \eqref{eq_form_cutoff_approximation} also shows that the
Dirichlet energy of $F$ equals the integral in
\eqref{eq_finite_energy}.
Finally, since $\chi_RF\in\mathcal{C}_c^\infty(\RR^n)$, the Poincar\'e
inequality \eqref{eq_1744} holds for $\chi_RF$. Passing to the limit as
$R\rightarrow\infty$, using the $L^2(\nu)$ convergence and the
Dirichlet-form convergence, shows that \eqref{eq_1744} also holds for $F$.
\end{proof}

\begin{lemma}
Let $G,S:\RR^n\rightarrow\RR$ be smooth functions such that, for some
$B<\infty$,
\begin{equation}
0\leq G\leq B,
\qquad
S\geq0,
\qquad
LG=2(S-G),
\qquad
\Gamma(G)\leq4BS.
\label{eq_cutoff_L_assumptions}
\end{equation}
Then $S\in L^1(\nu)$ and
\begin{equation}
\int_{\RR^n}LG\,d\nu=0.
\label{eq_integral_L_zero}
\end{equation}
\label{lem_integrating_L_cutoff}
\end{lemma}

\begin{proof}
Let $\chi_R$ be the functions from Lemma \ref{lem_energy_cutoffs}. Since
$\chi_R^2$ has compact support, the integration by parts \eqref{eq_integration_by_parts_L} gives
\begin{equation}
2\int_{\RR^n}\chi_R^2(S-G)\,d\nu
= \int_{\RR^n}\chi_R^2 L G\,d\nu
=
-2\int_{\RR^n}\chi_R\Gamma(\chi_R,G)\,d\nu.
\label{eq_trace_cutoff_identity}
\end{equation}
By \eqref{eq_cutoff_L_assumptions} and the pointwise
Cauchy-Schwarz inequality for $\Gamma$,
\begin{align*}
2\int_{\RR^n}\chi_R^2S\,d\nu
&\leq
2\int_{\RR^n}\chi_R^2G\,d\nu
+
4\sqrt{B}\int_{\RR^n}
\chi_R\sqrt{\Gamma(\chi_R)S}\,d\nu\\
&\leq
2B
+
\int_{\RR^n}\chi_R^2S\,d\nu
+
4B\int_{\RR^n}\Gamma(\chi_R)\,d\nu.
\end{align*}
Consequently,
$$
\int_{\RR^n}\chi_R^2S\,d\nu
\leq
2B+4B\mathcal{E}(\chi_R,\chi_R).
$$
Since $\chi_R\rightarrow1$ pointwise, Fatou's lemma and
\eqref{eq_energy_cutoffs} show that
$$
\int_{\RR^n}S\,d\nu\leq2B<\infty.
$$
We may now let $R$ tend to infinity in \eqref{eq_trace_cutoff_identity}. The
left-hand side converges to $2\int(S-G)\,d\nu$ by dominated convergence. On
the other hand,
\begin{align*}
2\left|
\int_{\RR^n}\chi_R\Gamma(\chi_R,G)\,d\nu
\right|
&\leq
2\mathcal{E}(\chi_R,\chi_R)^{1/2}
\left(
\int_{\RR^n}\chi_R^2\Gamma(G)\,d\nu
\right)^{1/2}\\
&\leq
4\left(
B\int_{\RR^n}S\,d\nu
\right)^{1/2}
\mathcal{E}(\chi_R,\chi_R)^{1/2}
\longrightarrow0.
\end{align*}
Therefore $\int(S-G)\,d\nu=0$, which is equivalent to
\eqref{eq_integral_L_zero}.
\end{proof}

\begin{lemma}
Assume that $\nabla\psi$ and $H=\nabla^2\psi$ are bounded in $\RR^n$. Then, for 
$i,j=1,\ldots,n$,
\begin{equation}
\int_{\RR^n}\psi_{ij}\,d\nu
=
\int_{\RR^n}\psi_i\psi_j\,d\nu.
\label{eq_first_euclidean_cutoff_identity}
\end{equation}
If, in addition, $\psi_{ijk}\in L^1(\nu)$, then
\begin{equation}
\int_{\RR^n}\psi_{ijk}\,d\nu
=
\int_{\RR^n}\psi_{ij}\psi_k\,d\nu
\label{eq_third_mean_identity}
\end{equation}
and
\begin{equation}
\int_{\RR^n}\psi_i\psi_j\psi_k\,d\nu
=
2\int_{\RR^n}\psi_{ijk}\,d\nu.
\label{eq_euclidean_cutoff_third}
\end{equation}
\label{lem_euclidean_cutoff_identities}
\end{lemma}

\begin{proof}
Fix a function $\zeta\in\mathcal{C}_c^\infty(\RR^n)$ such that
$0\leq\zeta\leq1$, with $\zeta=1$ on the Euclidean unit ball, and put
$\zeta_R(x)=\zeta(x/R)$. Since $\zeta_R$ has compact support, integration by
parts gives
$$
\int_{\RR^n}\zeta_R\psi_{ij}\,d\nu
= \int_{\RR^n}\zeta_R\psi_{ij} \, e^{-\psi} = 
\int_{\RR^n}\zeta_R\psi_i\psi_j\, d\nu
-
\int_{\RR^n}(\partial_i\zeta_R)\psi_j\, d\nu.
$$
The last integral tends to zero, since $\nabla\psi$ is bounded and
$\|\nabla\zeta_R\|_\infty=O(R^{-1})$. Since $H$ is bounded, the dominated
convergence theorem yields \eqref{eq_first_euclidean_cutoff_identity}.
Suppose now that $\psi_{ijk}\in L^1(\nu)$. Another integration by parts gives
$$
\int_{\RR^n}\zeta_R\psi_{ijk}\,d\nu
=
\int_{\RR^n}\zeta_R\psi_{ij}\psi_k\,d\nu
-
\int_{\RR^n}(\partial_k\zeta_R)\psi_{ij}\,d\nu.
$$
The last integral tends to zero, since $H$ is bounded. Letting $R$ tend to
infinity gives \eqref{eq_third_mean_identity}. Finally,
$$
\int_{\RR^n}\zeta_R\psi_i\psi_j\psi_k\,d\nu
=
\int_{\RR^n}\zeta_R
\left(\psi_{ij}\psi_k+\psi_j\psi_{ik}\right)d\nu
+
\int_{\RR^n}(\partial_i\zeta_R)\psi_j\psi_k\,d\nu.
$$
The last integral is $O(R^{-1})$, since $\nabla\psi$ is bounded. Letting
$R$ tend to infinity and using \eqref{eq_third_mean_identity}, first with
$(i,j,k)$ and then with $(i,k,j)$, proves
\eqref{eq_euclidean_cutoff_third}.
\end{proof}

\bigskip
\noindent Seminar for Statistics, Department of Mathematics, ETH Zurich,
8092 Zurich, Switzerland. \\
\textit{e-mail:} \texttt{yuansi.chen@stat.math.ethz.ch}

\bigskip
\noindent School of Mathematical Sciences, Tel Aviv University, Tel Aviv 6997801, Israel; and \\
Department of Mathematics, Weizmann Institute of Science, Rehovot 7610001, Israel. \\
\textit{e-mail:} \texttt{klartagb@tau.ac.il}

\end{document}